\documentclass[a4paper,UKenglish,cleveref, autoref]{lipics-v2019}

\usepackage{csquotes}

\title{Computability Aspects of Differential Games in Euclidian Spaces} 

\author{Gafurjan Ibragimov}{Universiti Putra Malaysia, Malaysia}{ibragimov@upm.edu.my}{https://orcid.org/0000-0002-4282-7482}{}

\author{Bakh Khoussainov}{University of Auckland, New Zealand \and \url{https://www.cs.auckland.ac.nz/~bmk/} }{bmk@cs.auckland.ac.nz}{}{}

\author{Arno Pauly}{Swansea University, UK \and \url{http://www.cs.swan.ac.uk/~cspauly/}}{arno.m.pauly@gmail.com}{https://orcid.org/0000-0002-0173-3295}{}

\authorrunning{B.~Khoussainov, G.~Ibragimov and A~Pauly}

\Copyright{Gafurjan Ibragimov, Bakh Khoussainov and Arno Pauly}



\category{}

\relatedversion{}

\supplement{}

\nolinenumbers 

\hideLIPIcs  

\EventEditors{John Q. Open and Joan R. Access}
\EventNoEds{2}
\EventLongTitle{42nd Conference on Very Important Topics (CVIT 2016)}
\EventShortTitle{CVIT 2016}
\EventAcronym{CVIT}
\EventYear{2016}
\EventDate{December 24--27, 2016}
\EventLocation{Little Whinging, United Kingdom}
\EventLogo{}
\SeriesVolume{42}
\ArticleNo{23}

\newtheorem*{lemma*}{Lemma}
\newtheorem*{theorem*}{Theorem}
\newtheorem*{question}{Question}

\begin{document}

\maketitle

\begin{abstract}
We study computability-theoretic aspects of differential games.
Our focus is on pursuit and evasion games played in Euclidean spaces in the tradition of Rado's \emph{Lion versus Man} game. In some ways, these games can be viewed as continuous versions of reachability games. We prove basic undecidability
of differential games, and study natural classes of pursuit-evasion games in Euclidean spaces where the winners can win via computable strategies. The winning strategy for Man found by Besicovitch for the traditional \emph{Lion versus Man} is not computable. We show how to modify it to yield a computable non-deterministic winning strategy, and raise the question whether Man can win in a computable and deterministic way.
\end{abstract}

\section{Introduction}

\vspace{-3mm}
\subsection{Background}
\vspace{-1mm}

In control theory the state of a dynamical system is given by a variable $x\in \mathbb R^n$ that evolves over time. The state
$x$  is described by an ordinary differential equation $\dot{x}(t)=f(t, x(t), u(t))$, where $t$ is a time parameter ranging in the interval $[0,T]$. The mapping $t\rightarrow u(t)$
is a control function. The goal is to find $u(t)$ from a class $C$ of control functions
that maximises

$$
\psi(x(T))-\int_0^T L(x(t), u(t), t) dt
$$
given initial condition $x(0)=x_0$, the running costs $L$ and the final payoff $\psi$. Differential games extend this model, where two (or more) players Player 0 and Player 1 play with competing goals.  The state of the system is still described by a differential equation $\dot{x}(t)=f(t, x(t), y(t), u_0(t), u_1(t))$, where $t\in [0,T]$ and $u_0(t)$, $u_1(t)$ are control functions for the players from function spaces $C_0$ and $C_1$. Given the initial condition, $\psi_i$ and $L_i$ as above, the goal of Player $i$, $i=0,1$, is to maximise
$$
\psi_i(x(T))-\int_0^T L_i(x(t), y(t), u_0(t), u_1(t), t) dt.
$$
Since both players can have an impact on the overall state of the system, one needs to specify what information is available to them, i.e.~what the functions $u_i$ can depend on besides the time. If no further information is available, we are dealing with \emph{open loop} strategies. If $u_i$ depends on $t$ and $x(t)$, we shall speak of \emph{positional} strategies (borrowing terminology from the area of sequential games played on finite graphs). Full information strategies would have $u_i(t)$ depend on $t$ and the entirety of $u_{0}(t')$ and $u_1(t')$ for $t' < t$.

R. Isaacs was the first who initiated the study of differential games \cite{isaacs}.
Pontryagin  continued the development of the theory of differential games in \cite{pontryagin2}.  These works
were motivated by possible
applications of pursuit-evasion games in missile guidance systems. Later a series of
research monographs in the area appeared, e.g. \cite{afriedman} \cite{subbotin}. Nowadays, the theory of differential games is an attractive topic of research ranging from control theory, economics, and geometry to robotics and optimisation.  See for instance, \cite{aubin}
 \cite{basar} \cite{dockner}.  The idea of pursuit and evasion  is a cornerstone of the theory of differential games.
 In this paper we focus on pursuit-evasion games played in Euclidean spaces and investigate computability-theoretic aspects of the games.
 For recent monographs we refer the reader to \cite{pachter} \cite{tsokos}.


 An informal set-up for the pursuit-evasion games is this.
Let $D$ be a domain in a Euclidean space. There are two players Pursuer and Evader.
Both players possess sets of control functions $C_P$ and $C_E$, respectively. We often call
control functions strategies. The players
use their control functions given that both move inside $D$.
The goal of Pursuer is to reach a vicinity of Evader, while the goal of Evader
is to stay away from Pursuer. More formally, Pursuer's  task is to find a control function $u \in C_P$ such that if Pursuer follows $u$ then Pusruer can reach a vicinity of Evader independent on control strategies from $C_E$ used by Evader. Similarly, Evader's task is to find a control function $v\in C_E$ such that if Evader follows $v$ then Pursuer will never be able to reach the vicinity of Evader independent on control strategies used by Pursuer.
This is a general set-up and many concepts, such as the concept of vicinity, conditions on control functions and domains
should be refined.

A classical example of pursuit and evasion game is the famous Man and Lion problem of R. Rado \cite{rado2}. R. Rado stated the problem in 1925 (see \cite{littlewood}).
A lion (Pursuer) and a man (Evader),  each viewed as a single point in a closed disc, have equal maximum speeds.
Can the lion catch the man if both move with equal speed and stay in the disk?
This had  been a well known problem since it was stated.
 S. Besicovitch  in 1952, after almost 30 years of the formulation of the problem and several wrong proofs (by other researchers),
  showed that the man can forever avoid from  being captured (see \cite{littlewood}). Croft \cite{croft} also studied the original problem in different settings.  Croft  proved that if the man runs along the curve of the disk, then the
 lion can capture the man. Furthermore,  Croft showed that in the $n$-dimensional Euclidean ball
 $n$ lions can catch the man while the man can escape from $n -1$ lions.  Later J. Flynn \cite{flynn} investigated the Man and Lion problem under various constrains, such as
the  lion has  a higher speed than the man, and gave some quantitative bounds on the strategies. Ivanov \cite{ivanov} studied the Man and Lion problem in compact spaces. More recent accounts of the problem in various settings are in \cite{alonso}, \cite{kopparty}, \cite{rote}, \cite{satimov},  \cite{sgall}, \cite{ibragimov}. For instance, Sgall in \cite{sgall} studies a discrete version of this problem in the positive quadrant of the plain and provides  an algorithm that extracts a winning strategy. S. Alexander, R. Bishop, and R.  Ghrist recently studied the Man and Lion type of problems on unbounded convex Euclidean domains \cite{alexander}.  Another interesting study is the work of K. Klein and S. Suri \cite{suri};
they investigate the question about the number of pursuers needed to capture an evader in  polygonal environments with
obstacles. We  mention Petrosyan \cite{petrosyan}, \cite{petrosyan2} who introduced parallel approach strategies and applied them to solve various pursuit problems in continuous setting. We too in this paper use parallel approach strategies, among several others such as   S. Besicovitch  type of strategies that we mentioned above.

There is a vast amount of research on decision-theoretic aspects of pursuit-evasion games played on graphs. Initial papers are by T.D. Parsons \cite{parsons}, \cite{parsons2}.   These works emphasise the problem of computing
 the number of pursuers needed to catch an evader. An interesting work devoted to the study of randomised strategies in pursuit-evasion games played on graphs is \cite{adler,adlerb}, where  the authors study the length of escape by Evader under various conditions put on the players and their randomised strategies. The paper \cite{fomin}
 investigates computational aspects (e.e. parametrized complexity)
 of pursuit-evasion games on graphs, and lists a comprehensive reference to  the topic.

\subsection{Contributions}
We start the investigation of computability-theoretic and complexity-theoretic aspects of differential games played in the finite-dimensional spaces $\mathbb R^n$.
We study pursuit-evasion games and focus on the following.

\noindent
{\bf 1}.
 In Section \ref{S:Undecidability}   we show that no algorithm exists that,
 given a differential game, decides the winner of the game (see Theorem \ref{Thm:undeicdable}).
The undecidability phenomenon already occurs in dimension $1$.
 Although undecidability  is expected in view of undecidability of differential equations
 and undecidability of reachability problem in dynamical systems
 \cite{hainry} \cite{denef},
 our result motivates the question about finding computable winning strategies
 in differential games.

\noindent
 {\bf 2}. In Section \ref{S:ComputableStrategies} we study existence of computable winning strategies in classical pursuit and evasion games. First, we consider the Pursuer and Evader  games in the whole space  $\mathbb R^n$ and in a half-space. In both cases, we show that
 the winners have winning strategies in cases when the spaces of control strategies $C_P$ and $C_E$ are universally bounded (see Proposition \ref{prop-space} and Theorem \ref{thm-half-space}). Second, we consider Pursuer and Evader games in
a convex compact set in the $n$-dimensional space
 $\mathbb R^n$ where
the  Pursuer's objective is to be within $\epsilon$-distance from Evader, where $\epsilon>0$ and the players have equal speeds.
Using Petrosyan's parallel pursuit strategies  \cite{petrosyan}, \cite{petrosyan2}, in Theorem \ref{epsilon-catch} we prove that Pursuer has a winning strategy computable almost everywhere. Moreover, all trajectories consistent with the winning strategy are computable. The proof uses a lemma showing that the projection operator on computable compact and convex sets is computable. In contrast, when $\epsilon=0$,
in $n$-dimensional compact convex sets, $n-1$ Pursuers cannot capture  Evader \cite{croft}, \cite{ivanov}.  The original proofs of these results (there are several of them
as we already mentioned) have no regards to computability, and one needs to give non-trivial arguments from computable analysis in order to obtain the results of this section. For instance, to our knowledge our proof of Theorem \ref{epsilon-catch} that uses projection operation is new.

\noindent
{\bf 3}.  In Section \ref{sec:L&MinDisk}, we focus on the classical Man and Lion problem in the disk, where both players have equal speeds. The goal of Pursuer is to achieve an exact catch. As we mentioned above Besicovitch proved that Evader has a winning control function in this game (see \cite{littlewood}). However, the winning strategy  constructed by Besicovitch is not computable because  the problem if a given point in $\mathbb R^2$ belongs to a line  is
 undecidable. In the proof of Theorem \ref{Thm:Algorithm} we recast  Besicovitch strategy, adapt it, and  provide a positive result
 by showing  that there exists a computable non-deterministic strategy such that Evader wins every play  consistent with the strategy. Our adaptation of Besicovitch strategy is not so obvious, and one needs to carefully reason and argue that the strategy obtained is computable (and non-deterministic).
We don't claim the existence of a computable deterministic strategy since our non-deterministic strategy is dependent on enumerations of input variables and different enumerations might give different outputs.


\subsection{Computable functions on $\mathbb R^n$}

We need to borrow some basic definitions from computable analysis \cite{weihrauchd} and computability \cite{soare2}. We use $\omega$ and $\mathbb Q$ to denote the set of natural numbers and rational numbers respectively. The Euclidean distance between $x, y\in \mathbb R^n$ is denoted by $|x-y|$.


Let $x$ be a point in the space $\mathbb R^n$. We start with the following definitions.
\begin{definition}\label{Dfn:Reps}
{\em A function $\phi: \omega \rightarrow  \mathbb Q^n$  {\em represents } the $x$ if $|x-\phi(n)|\leq 1/2^n$ for all
$n\in \omega$.}
\end{definition}

\noindent
Thus, a  presentation $\phi$ of point $x$ can provide arbitrary good approximation to $x$.
Clearly each point $x$ has infinitely many representations.

\begin{definition}
{\em The point $x$ is {\em computable} if there is a computable function $\phi: \omega \rightarrow  \mathbb Q^n$  representing $x$.}
\end{definition}
\vspace{-1mm}
\noindent
Since we seek strategies to be computable functions on reals or the space $\mathbb R^n$, we need to remind the reader the definition of a computable function:
\vspace{-1mm}
\begin{definition}
{\em A function $f:\mathbb X \rightarrow \mathbb R^m$, where $\mathbb X\subseteq \mathbb R^n$, is {\em computable} if there exists an oracle Turing machine $M$ such that for all points $x\in \mathbb X$, whenever the oracle tape of $M$ is fed with a representation $\phi$ of the  point $x$,  the machine $M$ outputs $y_n\in \mathbb Q^m$ such that
$|f(x)-y_n|< 1/2^n$ for all $n \in \omega$.}
\end{definition}

On the one hand, the most commonly used functions such as polynomials with computable coefficients,  and functions
$sin(x)$, $cos(x)$, $e^x$ are all computable. All computable functions are continuous \cite{weihrauchd}. Thus simple functions such as the step functions are not computable because they are discontinuous.


We will also use compact and convex sets. Among these sets we single out computable ones.
\begin{definition}{\em We call non-empty closed convex set $S\subseteq \mathbb R^n$ {\em computable}\footnote{Note that this is a non-standard terminology, which nevertheless serves well for our purposes.}
if the set of all rational points in $S$ is a decidable set and dense in $S$. }
\end{definition}

On computability for higher-type objects such as sets and function spaces, we refer to \cite{pauly-synthetic}. Non-computability in analysis is studied in the framework of Weihrauch degrees, see \cite{pauly-handbook} for a survey.


\section{Undecidability phenomenon}\label{S:Undecidability}


We start with 
one-dimensional differential games and
show that no algorithm exists that finds the winner in such games. \  On the real line $\mathbb R$ consider two players {\em Pursuer} and {\em Evader}.
Their initial positions are $x_0$ and $y_0$, and at time $t$ their positions are $x(t)$ and $y(t)$, respectively. The players move along the real line $\mathbb R$.  We assume that the positions satisfy the following ordinary differential equations with initial conditions:
\begin{center}{
\vspace{-1mm}
 $\dot{x}=u(t)$,  $x(0)=x_0$,\\
$ \dot{y}=v(t)$,  $ y(0)=y_0$,  $0<x_0<y_0$.}
\end{center}

The functions $u$ and $v$ are {\em control functions} taken over function spaces $C_P$ and $C_E$, respectively.
The positions $x(t)$ and $y(t)$ are evaluated as follows:
$$
x(t)=x_0+\int_0^t u(s) ds \ \ \ \mbox{and} \ \ \ y(t)=y_0+\int_0^t v(s) ds.
$$
In these equations Pursuer  follows the control function $u$ and Evader follows the control function $v$. 
 We would like to know if Pursuer catches  Evader, that is, if if there exists a time $t_c$ such that $x(t_c)=y(t_c)$.
A formal definition of the winner is this:
\begin{definition}{\em
{\em Pursuer wins  (or completes pursuit)} if there exists a strategy  $u \in C_P$ such that for all $v \in C_E$ there is a time $t_c$ that depends on $v$ such that $x(t_c)= y(t_c)$. {\em Evader wins (or evasion is possible)} if there exists a strategy $v \in C_E$ such that for all $u \in C_P$ we have $x(t) \neq y(t)$ for all $t$.}
\end{definition}
We note that $x(t)$ and $y(t)$ in the definition above both could depend on the strategies $u$ and $v$, if we either allow all strategies or at least positional ones.
An easy characterisation result  is the following.

\begin{claim} \label{Thm:1-dim}
Assume that Pursuer and Evader follow $u\in C_P$ and $v\in C_E$, respectively, with the initial conditions as above. Then Pursuer catches Evader if and only if there exists a time $t$ such that
$$
y_0-x_0 \leq \int_0^t (u-v) ds \ .\qed
$$
\end{claim}

Thus, we see that for neither player having access to information beyond the initial position is  relevant: Evader wants to move away from Pursuer as quickly as possible, and there is only one possible direction for that; dually, Pursuer wants to move in the direction of Evader, and Evader will always be in the same direction unless already caught. We can thus restrict to open-loop strategies in the one-dimensional case.

Note that if $C_E$ and $C_P$ are the sets of all computable functions on $\mathbb R$ then the  game has no winner. Thus, the existence of the winner depends on $C_E$ and $C_P$. Using Claim \ref{Thm:1-dim}, we have the following:

\begin{lemma}
\label{lemma:aux:function}
From a sequence $(a_n \in [0,1])_{n \in \mathbb{N}}$ we can compute a smooth\footnote{Recall that computing a smooth function $f$ means being able to evaluate $f^{(k)}(x)$ uniformly in $x$ and $k$.} function $f : \mathbb{R} \to \mathbb{R}$ such that
\begin{enumerate}
\item $f(n) = n - a_n$
\item The local maxima of $f$ are attained exactly on $\mathbb{N}$
\item $f' : \mathbb{R} \to [-1,5]$
\end{enumerate}
\begin{proof}
We first introduce some auxiliary functions. Let $h(x) = 0$ for $x \leq 0$ and $h(x) = e^{-\frac{1}{x}}$ for $x > 0$. We then define the usual smooth transition function $s$ as $s(x) = \frac{h(x)}{h(x) + h(1-x)}$. The function $s$ satisfies $s(x) = 0$ for $x \leq 0$, $s(x) = 1$ for $x \geq 1$, $s(x) \in (0,1)$ and $s'(x) \in (0,2)$ for $x \in (0,1)$. Let $b(x) = 0.5e^{\frac{-1}{1-4x^2}}$ for $x \in (-0.5,5)$ and $b(x) = 0$ else. The smooth function $b$ takes its unique maximum of $0.5$ at $0$, and satisfies $b'(x) \in [-1,1]$.

Now the function $f$ is defined as $f(x) = -0.5 + \sum_{n \in \mathbb{N}} b(x - n) + (1 - a_{n+1} + a_n)s(x - n)$. By construction of $s$ and $b$, only the terms corresponding to $\{n \in \mathbb{N} \mid n \leq x\}$ contribute non-zero terms to the sum, which ensures that $f$ is well-defined and smooth, and moreover, that $f$ inherits computability from $b$, $s$ and $(a_n)_{n \in \mathbb{N}}$.

As $s$ has no strict extrema, the local maxima are provided by the auxiliary function $b$ and are subsequently located exactly at $n \in \mathbb{N}$. When computing the derivative of $f$, we see that at most one term of the sum contributes. The contribution of $b$ is between $-1$ and $1$, and the contribution of the $s$-term is between $0$ and $4$ (since it is rescaled by a factor of up to $2$).
\end{proof}
\end{lemma}

\begin{theorem}\label{Thm:undeicdable}
Consider one-dimensional pursuit-evasion games. For any pair of computable sequences $(u_n)_{n \in \mathbb{N}}$, $(v_n)_{n \in \mathbb{N}}$ of strategies for Pursuer and Evader respectively, the set $E$ of those indices $n$ where Evader wins playing $v_n$ against Pursuer playing $u_n$ is $\Pi^0_2$, and there is a computable pair where $E$ is $\Pi^0_2$-complete.
\begin{proof}
That $E$ is always $\Pi^0_2$ follows from Claim \ref{Thm:1-dim}, taking into account that integration is a computable operation. To show completeness, we start with a Turing machine $\Phi$ and construct a computable function $u_\Phi$ such that Pursuer using $u_\Phi$ can catch an Evader that the constant function $v_c = 1$ iff $\Phi$ computes a total function. From that, we can construct the desired pair of sequences as $(u_c)_{n \in \mathbb{N}}$, $(v_{\Phi_n})_{n \in \mathbb{N}}$ using some standard enumeration $(\Phi_n)_{n \in \mathbb{N}}$ of Turing machines.

For $k \in \omega$, let $a_{k}$ be $2^{-s_k}$ if $\Phi(k)$ halts exactly at stage $s_k$, and let $a_{k}$ be $0$ if $\Phi(k)$ never halts. Note that the real number $a_k$ is computable from $k$ and $\Phi$. Let $f$ be a function computed from the sequence $(a_k)_{k \in \mathbb{N}}$ according to Lemma \ref{lemma:aux:function}. We want $f$ to be the trajectory of Pursuer. For that, we just set $u_\Phi = f'$, and let Pursuer start at $-2^{s_0}$ and Evader start at $0$. Since the local extrema of $f$ are attained on $\mathbb{N}$, we see that if pursuit is completed, it must hold that $f(k) \geq k$ for some $k \in \mathbb{N}$. Since $f(k) = k - a_k$, this happens iff $a_k = 0$, which in turn is equivalent to $\Phi(k)$ being undefined.
\end{proof}

\end{theorem}

\section{\hspace{-3mm} Extracting winning strategies}\label{S:ComputableStrategies}

\noindent

Given the undecidability Theorem \ref{Thm:undeicdable}, one way to study algorithmic aspects of differential games is to investigate if the winners have computable winning strategies. 
Our goal  is to investigate the problem of extracting computable winning strategies  in the classical Pursuer  and Evader
types of games.




We provide two types of results that are based on two possible definitions to the concept of vicinity. First, we consider exact catches in unbounded arenas such the full space and half-spaces. Here, we show that if the parameters are all computable, then either Pursuer can win using a computable positional strategy, or Evader can win using a computable open-loop strategy.
On the other hand, we consider compact, convex and computable sets $S$
in $\mathbb R^n$
and consider pursuit-evasion games in $S$ where the goal of Pursuer is to be within $\epsilon$-distance from Evader, where $\epsilon >0$. We show that in these games Pursuer has a positional winning strategy computable almost everywhere.

\subsection{\hspace{-4mm} Games played on $\mathbb R^n$ and half-spaces}

In this section we assume that the goal of Pursuer is exact catch, meaning that the position of Pursuer coincides with the position of Evader. Assume that the dynamics of Pursuer and Evader are given as: 
\begin{eqnarray}
\dot{x}=u, \  x(0)=x_0,  \
\dot{y}=v, \  y(0)=y_0, \    x_0\neq y_0. \label{1}
\end{eqnarray}
As mentioned if $C_E$ and $C_P$ consists of all computable functions, then there is no winner. We consider $C_E$ and $C_P$ consisting  of computable functions:
\begin{eqnarray}
C_P=\{u\mid |u| \leq a\} \ \mbox{and} \ C_E=\{v \mid |v|\leq b\},  \label{2}
\end{eqnarray}
where $a$ and  $b$ are computable reals. We assume that the players move in $\mathbb R^n$. Let us denote the game defined by $\Gamma(a,b, \mathbb R^n)$. As expected, we have:
\begin{proposition}\label{prop-space}
Consider the game $\Gamma(a,b, \mathbb R^n)$. If $a\leq b$ then Evader has an open loop computable winning strategy. If $a>b$ then Pursuer has a positional computable winning strategy.
\end{proposition}
\begin{proof}
Assume $b\geq a$. Let $e=(y_0-x_0)/|y_0-x_0|$. The winning strategy for Evader is $v=be$, that is, Evader moves along the line passing through $x_0$ and $y_0$ in the direction opposite of $x_0$ with velocity $b$. The strategy is computable, clearly.

Assume that $a>b$.  We construct a winning strategy for Pursuer. Let $\theta=|y_0-x_0|/ (a-b)$. Consider the following strategy:
\begin{center}{
\vspace{-1mm}
$u=(y_0-x_0)/\theta+v.$}
\end{center}
\vspace{-1mm}
 Also, the strategy is admissible since $|u|\leq a$. Indeed,
\begin{center}{
\vspace{-1mm}
$|u|^2=|y_0-x_0|^2/\theta+(2/\theta)(y_0-x_0, v) +|v|^2\leq 
(a-b)^2+2(a-b)b +|v|^2\leq a^2$}
\end{center}

The pursuit can be completed by time $t=\theta$ no matter what strategy $v(t)$ Evader uses. This shows that $u$ is a winning strategy. Indeed,
$$
x(\theta)=x_0+ \int_0^{\theta} [(y_0-x_0)/\theta+v] \ dt = y_0+\int_0^{\theta} v(t)dt= y(\theta).
$$
Finally, note that the strategy $u$ is computable.
\end{proof}
More interesting case is when the players play inside of a half-space. Let $P$ be a hyperplane with a computable dense subset of rational points in it. The hyperplane divides $\mathbb R^n$ into two half-spaces. 


The game is defined as follows. The dynamics of Pursuer and Evader are given by equations as in (1),
and the sets of control functions are given as in (2) such that computable points  $x_0, y_0$ belong to the same half-space, say $H$.
The players move inside the half-space $H$.
We denote the game by $\Gamma(a,b, H)$. An interesting observation is that if $a>b$, then the strategy $u$ given in Proposition \ref{prop-space} might be illegitimate because Pursuer using the strategy might leave the space $H$. Therefore, Pursuer
needs to have more sophisticated winning strategy.

\begin{theorem}\label{thm-half-space}
Consider the game $\Gamma(a,b,  H)$. If $a\leq b$ then Evader has an open loop computable winning strategy. If $a>b$ then Pursuer has a positional computable winning strategy.
\end{theorem}

\begin{proof}
Without loss of generality we assume that $x_0=0$.
Assume that $b\geq a$. Let $e$ be a computable vector of length $1$ parallel to $H$ such that if Evader moves along the ray
$y_0+te$, with $t\geq 0$, the distance from $x_0$ increases with increasing $t$. Consider the strategy $v$ for Evader defined as $v=eb$. This is a computable (and open loop) strategy. Moreover, it is not too hard to see that it is a winning strategy.


Let $a>b$. To simplify notations, assume that the space is $\mathbb R^2$.  Let $y(t)$ be a position of Evader. Define two unit vectors: the first is $e(t)=y(t)/|y(t)|$, and  the second is $n(t)$ the vector orthogonal to $e(t)$. Let $v$ be Evader's strategy.  At time $t$,   compose $v(t)$ in two components based on
$e(t)$ and $n(t)$:
$$
v(t)=v_{e}(t) e(t)+ v_{n}(t) n(t).
$$
We now define the strategy $u$ for Pursuer. The strategy is written with respect to the system $\{e(t), n(t)\}$:
$$
u(t)=u_{e}(t) e(t)+ u_{n}(t) n(t), \  \  \mbox{where}
$$
$u_n(t)$=$(|x(t)|/|y(t)|) v_n(t)$ and $u_e(t)$=$\sqrt{a^2-u_n(t)^2}$. It is clear that $|u|\leq a$.


Note that the coordinates $v_e(t)$ and $u_e(t)$ can be expressed through the derivatives of the distance functions $dist_E(t)=|y(t)|$ and $dist_P(t)=|x(t)|$:
$$
dist^{\prime}_E(t)=v_e(t) \ \mbox{and} \ dist^{\prime}_P(t)=u_e(t).
$$
To show that $u$ is a winning strategy, we prove that for any strategy $v$ of Evader there is a $t_v$ such that $dist_P(t_v)=dist_E(t_v)$. For that we compute the difference $d(t)=dist_E(t)-dist_P(t)$. Since  $u_n(t)|=\frac{|x(t)|}{|y(t)|}\leq |v_n(t)|$ we get:
\begin{center}{
$d(t)= (|y_0|+\int_0^tv_e(s)ds) -(|x_0|+\int_0^tu_e(s)ds)
\leq  |y_0|+\int_0^t \sqrt{b^2-v_{n}(s)^2}ds -\int_0^t\sqrt{a^2-u_{n}(s)^2}ds
\leq  |y_0|+\int_0^t \sqrt{b^2-v_{n}(s)^2}ds -\int_0^t\sqrt{a^2-v_{n}(s)^2}
\leq |y_0| - \int_0^t  \frac{(a^2-b^2)}{\sqrt{b^2-v_{n}(s)^2} + \sqrt{a^2-v_{n}(s)^2}}$
}
\end{center}
Hence we conclude the following:
$$
d(t)\leq |y_0|-\int_0^t(a-b)ds=|y_0|-(a-b)t
$$
Initially $d(0)>0$, and at time $T=|y_0|/(a-b)$ we have $d(T)\leq 0$.  Since $d(t)$ is continuous,
Pursuer catches Evader  between times $0$ and $T$. Note that the strategy $u$ is a computable strategy.
\end{proof}

\subsection{Games with $\epsilon$-catch}

The set-up of the game we investigate is the following. The dynamics of Pursuer and Evader are given by the equations:
\begin{eqnarray}
\dot{x}=u,  \  x(0)=x_0,   \
\dot{y}=v, & y(0)=y_0,        x_0\neq y_0, \label{1}
\end{eqnarray}
where $x_0$, $y_0$ are computable points,
and  $u$, $v$ are strategies  taken from spaces $C_P$ and $C_E$ that consist of
Borel functions (which are not necessarily continuous)  with $|{u(t)}| \leq 1$, $|v(t)|\leq 1$ for $t \ge 0$.
The players play inside a compact computable convex set $S\subset \mathbb R^n$.
So, $x(t), y(t) \in S$ for all $t\geq 0$.


Let $\epsilon> 0$ be  a rational number. The goal of Pursuer is to reach the $\epsilon$-neighbourhood of Evader.
We assume that $|x_0-y_0|> \epsilon$. At all  times $t$ the players know the velocities $u(t)$, $v(t)$ and positions  $x(t)$, $y(t)$.


We denote this Pursuer-Evasion game by $\Gamma(S, \epsilon)$. A  definition of winning in game $\Gamma(S, \epsilon)$ is this:
\begin{definition} {\em
{\em Pursuer wins (or completes pursuit in) the game $\Gamma(S, \epsilon)$} if there exists a strategy $u \in C_P$ such that for all $v \in C_E$ there exists a time $t_v$ such that $|x(t_v)- y(t_v)| \leq \epsilon $. Similarly,  {\em Evader wins the game (or evasion is possible)} if there exists a strategy $v \in C_E$ such that for all $u \in C_P$ we have $|x(t)-y(t)|> \epsilon $ for all $t$.}
\end{definition}

To study $\Gamma(S, \epsilon)$ we use two auxiliary games.
The winning strategies in these two  games will be needed to construct the strategies with nice computability-theoretic properties
in the original game $\Gamma(S, \epsilon)$.

Let $u$ be Pursuer's strategy in a differential game.  The strategy depends on $(x,y,v)$, where $x$ is Pursuer's position, $y$ is Evader's position, and $v$ is the velocity of Evader at  $y$. Assume that Pursuer follows $u$. If Evader follows a strategy $v(t)\in C_E$, then  Pursuer's velocity at time $t$ equals $u(x(t), y(t), v(t))$. Hence, the position of Pursuer at point $t$ equals:
$$
x_{u,v}(t)=x_0+\int_{0}^t u(x(s), y(s), v(s))ds.
$$
\begin{definition}
{\em The functions $x_{u,v}(t)$ defined above are called {\em trajectories (of Pursuer)} consistent with $u$ in the given game.}
\end{definition}
\begin{theorem} \label{epsilon-catch}
Let $\Gamma(S, \epsilon)$ be the game described above. Pursuer has a winning strategy $u$ such that for every continuous $v\in C_E$ the trajectory $x_{u,v}$ is a function computable relative to $v$.
\end{theorem}

\begin{proof} For the proof we introduced two auxiliary differential games and study their winning strategies.

\noindent
{\bf The first auxiliary game}.
Let $l$ be a ray starting at $x_0$ and passing through $y_0$. The constants that define the ray $l$ are computable
since $x_0$ and $y_0$ are computable.  Let $H$ be a {\em supporting  hyperplane} for $S$ that goes through $l$, that is, $H$ is defined as follows:
(1) $S\cap H\neq \emptyset$,  (2) $H$ does not contain an interior point of $S$, and (3) $H$ is perpendicular to $l$.  Such $H$ exists (in fact, there are two of them and we select one) because $S$ is compact and convex. The hyperplane divides $\mathbb R^n$ into two half-spaces.
Let $X$ be the closure of the half-space that contains $S$.
The half-space $X$ is computable.

\begin{figure}[h]
 \includegraphics[width=0.5\textwidth]{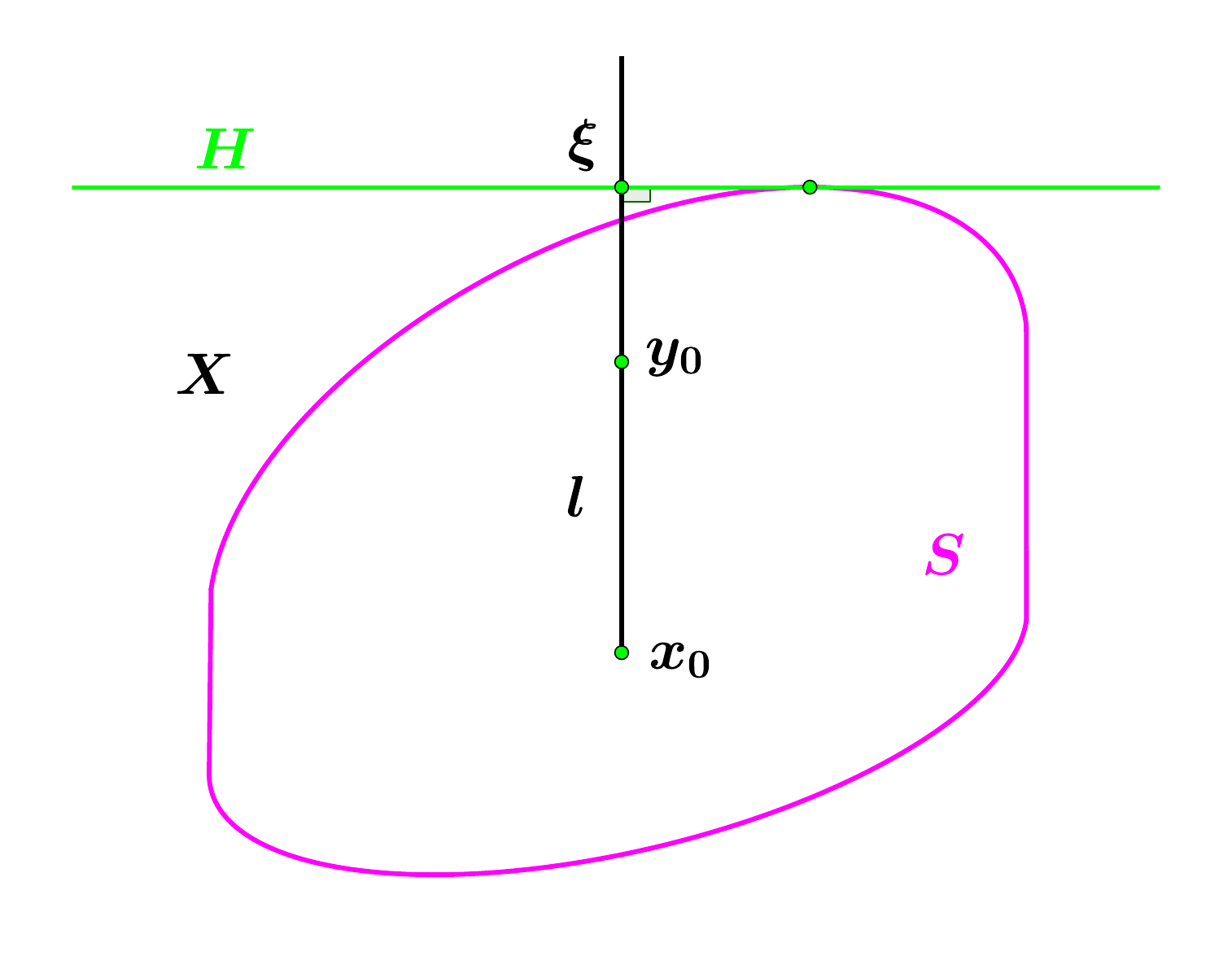}
 \caption{Game in the half space X.}
 \label{F1}
\end{figure}


Let $\xi$ be the intersection point of $H$ and $l$. This is a computable point. Consider the set of all points $\eta$ in the hyperplane $H$ such that \
$|\eta-x_0|=|\eta-y_0| + \epsilon$.  The distance $\theta=|\eta-y_0|$ is a computable distance; for instance, it can be found through solving the equation
$$
(\theta-\epsilon)^2-|\xi-x_0|^2=\theta^2-|\xi-y_0|^2.
$$
The first auxiliary game is defined as follows:
\begin{enumerate}
\item The players are Pursuer $z_1$ and Evader $y$. The dynamics  are
\begin{center}{
$\dot{z}_1=w,  \dot{y} =v$, $v\leq 1$,     $|w|\leq 1+\epsilon/\theta=\rho$,   $z_1(0)=x_0$,  $y(0)=y_0$,}
\end{center}
where $v$ and $w$ are Borel-measurable functions.
\item Both players, Pursuer $z_1$ and Evader, play inside the half-space $X$ that is defined above.
\item Pursuer $z_1$ wants to ensure an exact catch.
\end{enumerate}
\noindent
Pursuer $z_1$ has a computable winning strategy in the first auxiliary game defined as
$$
w_1(t)=v(t)-(v(t), e(t))e+e\sqrt{\rho^2-1+(v(t), e)^2}.
$$

\medskip
\noindent
{\bf The second auxiliary game}. The definition of this game is as follows:
\begin{enumerate}
\item The players play inside $S$.  The players are Pursuer $z_2$ and Evader $y$.
\vspace{-1mm}
\item The dynamics of the game are
 \vspace{-1mm}
 \begin{center} {
 $\dot{z}_2=w$,  $\dot{y}=v,$ \  \  $z_2(0)=x_0$, \ $y(0)=y_0$,\\
\  $|v| \leq 1$, \   $|w|\leq 1+\epsilon/\theta=\rho$,}
\end{center}
\vspace{-1mm}
where $x_0$ and $y_0$ are as in $\Gamma(S, \epsilon)$.
\vspace{-1mm}
\item
Pursuer $z_2$  wants to ensure exact catch.
\end{enumerate}
\begin{lemma}\label{Lem:2nd-game}
In the second auxiliary game Pursuer $z_2$ has a winning strategy such that
every trajectory consistent with the strategy is computable from Evader's strategy.
\end{lemma}
To prove the lemma,  we  use Pursuer's winning  strategy $w_1$ from the first auxiliary game. The main idea is to project the winning strategy of the first auxiliary game onto the second. For this we consider the projection function $P(z)$, $P:\mathbb R^n \rightarrow S$ which is defined by the equation:
$$
|P(z)-z|=\min\limits_{\xi\in S}|\xi-z|.
$$

\begin{claim} \label{clm:projection}
The function $P(z)$ is correctly defined and computable.
\end{claim}
The value $P(z)$ is correctly defined because $S$ is compact and convex. We show that $P(z)$  is
computable. Uniformly on $z$,
for every rational number $\epsilon>0$ one can effectively find a rational point $x_{\epsilon}$ and a ball $B_{r(\epsilon)}(z)$ of radius $r(\epsilon)$ with center $z$ such that
\begin{enumerate}
\item $B_{r(\epsilon)}(z) \cap S\neq \emptyset$. So,  $B_{r( \epsilon)}(z) \cap S$ contains $P(z)$.
\item We have $|P(z)- x_{\epsilon}|\leq \epsilon$. 
\end{enumerate}


All these can computably be done because $S$ is convex and compact set with the property that all rational points of $S$ is a dense and computable set.



The  function $P(z)$ has the following properties:
\begin{enumerate}
\item[P1.] For all $z'\in S$ we have $P(z')=z'$.
\item[P2.] $|P(z')-P(z'')| \leq |z'-z''|$ for all $z', z'' \in \mathbb R^n$.
\item[P3.] For any trajectory $x_{w_1,v}(t)$ of Pursuer in the first auxiliary game,  the function $P(x_{w_1,v}(t))$ is computable relative to $v$, where $t\in [0,\theta]$.
\end{enumerate}
For Property 3, recall that $w_1$ is the winning computable strategy in the first auxiliary game. So, the function, $P(x_{w_1,v}(t))$ is computable by Claim \ref{clm:projection}.

Assume that the trajectory $x_{w_1,v}(t)$ is consistent with $v\in C_E$ and $w_1$. Pursuer's $z_2$ strategy in the second auxiliary game is then  defined through the equation $z_2(t)=P(x_{w_1,v}(t))$. Note that $z_2$ is computable from $x_{w_1,v}$ and thus from $v$.
We need to show that for $z_2(t)$ we have $|\dot{z}_2| \leq \rho$.  Indeed:
\begin{eqnarray*}
|\dot{z}_2(t)|=
 \lim\limits_{h\rightarrow 0^+}\frac{|z_2(t+h)-z_2(t)|}{h}=\\
= \lim\limits_{h\rightarrow 0^+}\frac{|P(x_{w_1,v}(t+h))-P(x_{w_1,v}(t))|}{h}\le \\
 \le \lim\limits_{h\rightarrow 0^+}\frac{|x_{w_1,v}(t+h)-x_{w_1,v}(t)|}{h}=|\dot{x}_{w_1,v}(t)|\leq \rho.\\
\end{eqnarray*}
Now we show that Pursuer $z_2$ can complete a pursuit. Indeed, let $\tau $ be time at which Pursuer $z_1$ captures Evader in the first auxiliary  game. Since $y(\tau)\in S$ we have $z_1(\tau) \in S$. Then by Property 1, above we have $z_2(\tau)=P(x_{w_1,v}(\tau))=y(\tau)$. We proved the lemma.


Now we prove the theorem. Let $w_2$ be the winning strategy for Pursuer described in the second auxiliary game. Note that this is not necessarily computable strategy. Using $w_2$, we show that Pursuer has a desired winning strategy in the game $\Gamma(S, \epsilon)$.


 We claim that Pursuer's desired strategy $u$ is defined by $u(t)=(1/\rho)w_2(t)$, where $0\leq t \leq \theta$. Indeed, we
 note that $|u(t)|\leq 1$. Now for any Evader's strategy $v$ in $\Gamma(S,\epsilon)$ we reason as follows.  Firstly, we show that $x_{u,v}(t)\in S$.  For this, note that
$$
x_{u,v}(t)=x_0+\int_0^t \frac{1}{\rho} \ w_2(s)ds= \left(1-\frac{1}{\rho}\right)x_0+\frac{1}{\rho} \ x_{w_2,v}(t).
$$
\noindent
Since $1/\rho<1$, by convexity we have $x_{u,v}(t)\in S$. Secondly, the distances between the trajectories $x_{w_2,v}(t)$ (in the second auxiliary game) and $x_{u,v}(t)$ are uniformly bounded by $\epsilon$, where $0 \le t \le \theta$. To see this, consider the following calculations:
\begin{eqnarray*}
|x_{w_2,v}(t)-x_{u,v}(t)|=\hspace{20mm} \\
=\left |x_0+\int\limits_0^t \dot{x}_{w_2,v}(s)ds-x_0-\int\limits_0^t \frac{1}{\rho}\dot{x}_{w_2,v}(s)ds\right|\\
=\frac{\rho-1}{\rho}  \left|\int \limits_0^t \dot{x}_{w_2,v} (s) ds   \right | \leq \frac{\rho-1}{\rho}  \int \limits_0^t |\dot{x}_{w_2,v} (s) | ds \\
\leq  \frac{\rho-1}{\rho}\rho \cdot t=\frac{\epsilon}{\theta} \cdot t\leq \epsilon. \hspace{17mm}
\end{eqnarray*}

Finally, since $z_2(\tau)=y(\tau)$ we get $|y(\tau)-x_{u,v}(\tau)|=|x_{w_2,v}(\tau)-x_{u,v}(\tau))|\leq \epsilon $. Thus, Pursuer is able to be within $\epsilon$-distance from Evader. We conclude that $u$ is a winning strategy.


The only thing left is to note that the trajectories $x_{u,v}(t)$ are all computable from $v$ because of Lemma \ref{Lem:2nd-game} and Property (P3) of the projection function that we stated above.
\end{proof}
We note that our proof does not imply that $u$ is a computable winning strategy.
The strategy $u$ is discontinuous on the set $\partial(S)$ of boundary points of $S$. However, at all other points $S\setminus \partial(S)$ the strategy $u$ is computable. Note that $\partial(S)$ has  measure $0$. Hence, we have the following corollary:

\begin{corollary}
In the game $\Gamma(S, \epsilon)$ Pursuer has a winning strategy $u$ that is computable almost everywhere in $S$. \qed
\end{corollary}

\section{Pursuit and evasion in a  disk}\label{sec:L&MinDisk}
In this section Pursuer and Evader  move inside the 2-dimensional closed disk of radius $R$ with centre $O$ at $(0,0)$, where $R$ is a computable number. Both players move with velocity at most  $1$.  
Pursuer wants to catch Evader in finite time, and Evader wants to avoid being captured. Say that Evader {\em wins} the game
if Evader has a strategy $v$ such that for all strategies $u$ of Pursuer $x(t)\neq y(t)$ for all times $t$, where $x$ and $y$ are consistent with $u$ and $v$, respectively.

The winning strategy constructed by Besicovitch for these games is, as we mentioned, not computable. It involves the undecidable problem of testing if a real point $x$ belongs to a real line.  We will show how to modify this strategy to obtain a computable one. However, this requires us to introduce the notion of a {\em non-deterministic strategy}. Informally, a non-deterministic strategy does not completely specify the behaviour of the player, but any deterministic behaviour consistent with the strategy will be winning. This is a common phenomenon in computable analysis. Often we exhibit an algorithm to compute some solution, but the solution of the algorithm  depends on the specific encoding of the input.  A classic example is that given a non-constant polynomial, we can compute some complex root of it, but there is no computable choice function for this \cite{weihrauchd}.

We construct a strategy for Evader that acts in the following way. Initially, Evader selects a direction $v_0$ and a time delay $\Delta_0$ based on the current position of Evader and Pursuer $p_0$. Then Evader moves in direction $v_0$ for time $\Delta_0$. At time $\Delta_0$, Evader observes again the position $p_1$ of Pursuer, and selects direction $v_1$ and time $\Delta_1$, and this continous on. We call strategies of this form \emph{piecewise-blind}.  To be a winning strategy, Evader needs to ensure that provided that $d(p_{t},p_{t+1}) \leq \Delta_t$, the position $e_t$ of Evader at time $t$ never coincides with $p_t$, and that $\sum_{t \in \mathbb{N}} \Delta_t$ diverges to infinity. Such a strategy is a computable non-deterministic, if $\Delta_t$ is selected by a computable multivalued function of $t, p_t, e_t$. If $\Delta_t$ can even be chosen as a computable single valued function, we might still have a computable non-deterministic strategy.


\begin{theorem} \label{Thm:Algorithm}
Consider the pursuit and evasion game in the disk as described above.
There exists a computable non-deterministic piecewise-blind winning strategy for Evader that works for any given initial position.
\end{theorem}
\begin{proof}
\noindent
 We slightly change S. Besicovitch's strategy and explain that the new strategy is a desired one.

If Evader is on the border, then she can make a move inside without being caught (by moving strictly less than half its current distance to Pursuer). We can thus assume that Evader is inside the disk.  Let $r$ be the distance from the position $E_0$ of Evader to the border of the disk
at time $t_{0} =0$; the value $r$ is our parameter, and it is computable from the positions.


Let $E_t$ and $P_t$ be positions of Evader and Pursuer  at time $t$, respectively.
We change directions of the velocities of Evader
at times $t_0, t_{1} ,t_{2} ,\ldots ,t_{n} ,\ldots $, where $t_0=0$, $E_{t_0}=E_0$, and $t_{i} $ will be  the times at which the distance from $E_{t_{i}}$
to the circumference of the border of the disk equals $r/(i+1)$.

To describe Evader's desired strategy, we need some notations.
Let $A_i$ be the line from the centre of the disk  $O$ to the position $E_{t_i}$ and $B_i$ be the line perpendicular to $A_i$ going through $E_{t_i}$. These two lines  are computable from $E_{t_i}$.   On line $B_i$, we choose two points $E_{lt_{i}}$ and $E_{rt_{i}}$
such that
$$
|OE_{lt_{i}}|=|OE_{rt_{i}}|= R- r/(i+2).
$$
Geometrically, we can view these points as left and right of $E_{t_i}$ (respectively) on the line $B_i$; that is why we use the indices $l$ and $r$. Thus, the distance from the positions $E_{lt_{i}}$ and $E_{rt_{i+1}}$ to the border of the disk is $r/(i+2)$. By induction, we know that  $|OE_{t_i}| =R- r/(i+1)$.
 Construct the balls
$B(E_{lt_{i}}, |E_{t_i}E_{lt_{i}}|)$  and $B(E_{rt_{i}}, |E_{t_{i}} E_{rt_{i}}|)$, where $B(x, q)$ is for the ball of radius $q$ and centred at $x$.
By construction, these two balls have only one point  in common, which is $E_{t_i}$.  Hence, the current position $P_{t_i}$
of Pursuer  does not belong to one of these balls unless $P_{t_i}=E_{t_i}$. Below, we will show that  $P_{t}=E_{t}$  is impossible. Hence, given an approximation to $P_{t_i}$ we can compute a ball that does not contain $P_{t_i}$.
We now construct the following strategy for Evader. Select a ball among the balls
$B(E_{lt_{i}}, |E_{t_i}E_{lt_{i}} |)$ and $B(E_{rt_{i}}, |E_{t_{i}} E_{rt_{i}}|)$ that does not contain $P_{t_i}$. It is this choice which is by necessity non-deterministic. If
$P_{t_i} \notin B(E_{lt_{i}}, E_{t_i}E_{lt_{i}})$ at the time $t_i$ is confirmed first, then Evader moves towards $E_{lt_{i}}$ with the velocity 1,
and if $P_i \notin B(E_{rt_{i}}, E_{t_i}E_{rt_{i}})$ is confirmed first, then Evader moves towards $E_{rt_{i}}$ with the velocity 1.
Set $E_{t_{i+1}}=E_{lt_{i}}$ in the first case, and $E_{t_{i+1}}=E_{rt_{i}}$ in the second.



Now we prove that the strategy described is a winning strategy for Evader. The proof is a beautiful argument given by
S. Besicovitch described  in  \cite{littlewood} but for completeness we provide the proof  in two steps.

\noindent
{\bf Step 1}.  Here we prove that at each time period from $t_i$ to $t_{i+1}$ evasion is possible. Assume that at time $t_i$ Pursuer position $P_{t_i}$ does not belong to the ball $B(E_{t_{i+1}}, |E_{t_i}E_{t_{i+1}} |)$, and that Evader moved towards $E_{t_{i+1}}$.
This implies that  $|P_{t_i}E_{t_{i+1}}|> |E_{t_i}E_{t_{i+1}}|$. To arrive to a contradiction, let $\tau$ be a time such that $E_{\tau}=P_{\tau}$ and $t_i\leq \tau \leq t_{i+1}$.  From time $\tau$ on Pursuer can move together with Evader at velocity $1$ towards position $E_{t_{i+1}}$. Hence, $|P_{t_i}E_{t_{i+1}}|\leq |E_{t_i}E_{t_{i+1}}|$;  this is a contradiction as
$P_{t_i}$ does not belong to  the ball $B(E_{t_{i+1}}, |E_{t_i}E_{t_{i+1}} |)$.

\noindent
{\bf Step 2}. Here we prove that evasion is possible. Let $T_i$ be  time taken by Evader to arrive at positions $E_{t_{i+1}}$  from position $E_{t_i}$.  It is clear that $T_{i} =|E_{t_i}E_{t_{i+1}}|$ since the velocities equal to $1$. We estimate the time $T_i$ as follows:

\begin{equation} \label{1}
T_{i} =  \sqrt{\left(R - \frac{r}{i+2} \right)^{2} - \left(R - \frac{r}{i+1} \right)^{2}} \ge \frac{r}{i+2}.
\end{equation}
Hence
\[T_{1} + T_{2} + \ldots \ge \frac{r}{2} +\frac{r}{3} +\frac{r}{4} +\ldots = r\cdot \sum _{n=2}^{\infty }\frac{1}{n}  =\infty ,\]

\noindent since the series $\sum _{n = 2}^{\infty }(1/n)$ is a divergent series. Thus, this strategy is  winning for Evader.
\end{proof}

\subsection{Making the strategy smooth}
The non-deterministic computable piecewise-blind strategy we construct in Theorem \ref{Thm:Algorithm} ensures that we can compute the position of Evader at any time from the given parameters (including the strategy Pursuer is using). However, Evaders velocity is not even continuous, as she is abruptly changing direction at the decision points. This defect is easily remedied, and in fact, we can make the trajectories computably smooth, meaning that all derivatives of the position of Evader are not only well-defined, but computable\footnote{It is well-known that in general a function $f : \mathbb{R} \to \mathbb{R}$ can be computable and differentiable, but can have $f'$ computable the Halting problem.}.

\begin{lemma}
If Evader has a piecewise-blind winning strategy for any given initial condition such that the chosen time delay depends only on the current distance of Evader to the boundary, then Evader has a smooth winning strategy which can be computed from the piecewise-blind strategy.
\begin{proof}
We turn the given strategy into a smooth strategy by iteratively removing the sharp turns taken at the decision times. Assume that the original winning strategy prescribes moving to some point $x$ as the next step taking up time $t$. As the strategy is winning we know that Pursuer cannot reach $x$ within time $t$ from her current position. By compactness, there is some closest distance $d$ to $x$ she can realize. Now any alternative trajectory Evader could take that is never more than $\frac{d}{2}$ away from the prescribed linear movement is also safe up to time $t$. Whatever Evaders current momentum is, he can extend his past trajectory in a smooth way staying $\frac{d}{2}$-close to the prescribed movement up to time $t$, and safely reaching a point $x'$ with the same radial distance as $x$.

We then repeat the very same argument again and again. As the time delays of the original strategy depend only on the radial distance, the sequence of time delays we are obtaining in this way are also consistent with the original winning strategy. This in particular implies that there sum diverges. As the resulting smooth strategy is safe at any time during its construction, it is also a winning strategy.
\end{proof}
\end{lemma}

\begin{corollary}
Evader has a computable non-deterministic winning strategy for any given initial condition such that any compatible trajectory of Evader is smooth.
\end{corollary}

\subsection{Evasion in a disk: can a winning strategy be deterministic?}
We proved that Evader has a non-deterministic computable winning strategy in pursuit and evasion game that occurs in a disk. The following question is immediate:
\begin{question}
\label{que:deterministic}
Does Evader have a computable deterministic winning strategy?
\end{question}

The answer to this question has eluded us.  A priori, a natural approach would be to adapt Besicovitch' strategy further by replacing our non-deterministic case distinction of whether Evader should move clockwise or counter-clockwise for the next time period by a region of continuous interpolation. While we do not formulate a formal impossibility result, we shall nevertheless explain the problem encountered on this undertaking.

Assume that Pursuer is closer to the center of the disk than Evader, and that Pursuer is inside the region where we continuously interpolate between the two options suggested by Besicovitch' strategy. The continuous interpolation affects both the direction in which we move, and the time delay until the next choice. We fix the current distance between Pursuer and Evader, and consider how these vary by the precise position of Pursuer on the relevant circle. Since a time delay of $0$ does not yield a well-defined strategy, by compactness we have a minimum positive time delay.

If this minimum time delay is small, we face a problem with the requirement that the sum of all time delays must diverge to infinity for the strategy being well-defined, not matter how Evader acts. We could attempt to allow for very small time delays by moving around the region of continuous interpolation in such a manner that if Pursuer actually is in a position causing a very small time delay at one step, he could not also be at a region of very small time delay at the subsequent step (since he has not much time to get to the new region). How far apart we can keep these regions of interpolation depends on how close Pursuer is to Evader -- and Besicovitch' strategy allows Pursuer to catch up too rapidly to provide us with the required flexibility.

Consider now the case that the minimum time delay is large. There is a position of Pursuer that causes Evader to move directly towards the boundary of the disk. Since the distance of Evader to the boundary of the disk is the resource Evader is spending in exchange for time in Besicovitch' strategy, this must not happen too often. Together with the preceding paragraph, this seems to leave us no room for the intended continuous interpolation.

To rule out a computable deterministic winning strategy, the analysis presented above is insufficient, and new arguments would need to be found. On the other hand, it might be that a different approach than Besicovitch' strategy more readily yields a computable deterministic strategy. Engaging further with the question should shed new light on this classic game.

\bibliography{../spieltheorie}

\begin{thebibliography}{10}

\bibitem{adlerb}
Micah Adler, Harald R\"{a}cke, Naveen Sivadasan, Christian Sohler, and Berthold
  V\"{o}cking.
\newblock Randomized pursuit-evasion in graphs.
\newblock In {\em Proceedings of the 29th International Colloquium on Automata,
  Languages and Programming}, ICALP '02, pages 901--912. Springer, 2002.
\newblock URL: \url{http://dl.acm.org/citation.cfm?id=646255.758840}.

\bibitem{adler}
Micah Adler, Harald R\"acke, Naveen Sivadasana, Christian Sohler, and Berthold
  V\"ocking.
\newblock Randomized pursuit-evasion in graphs.
\newblock {\em Combinatorics, Probability \& Computing}, 12(3):225--244, 2003.
\newblock \href {http://dx.doi.org/10.1017/S0963548303005625}
  {\path{doi:10.1017/S0963548303005625}}.

\bibitem{alexander}
S.~Alexander, R.~Bishop, and Robert Ghrist.
\newblock Capture pursuit games on unbounded domain.
\newblock {\em L\"enseignement Math\"ematique}, 55(1-2):103--125, 2009.
\newblock \href {http://dx.doi.org/10.4171/LEM/55-1-5}
  {\path{doi:10.4171/LEM/55-1-5}}.

\bibitem{alonso}
Laurent Alonso, Arthur~S. Goldstein, and Edward~M. Reingold.
\newblock ``{L}ion and {M}an'': {U}pper and lower bounds.
\newblock {\em ORSA Journal on Computing}, 4(4), 1992.
\newblock \href {http://dx.doi.org/10.1287/ijoc.4.4.447}
  {\path{doi:10.1287/ijoc.4.4.447}}.

\bibitem{aubin}
Jean-Pierre Aubin.
\newblock {\em Optima and equilibria}.
\newblock Springer, 2nd edition, 1998.
\newblock \href {http://dx.doi.org/10.1007/978-3-662-03539-9}
  {\path{doi:10.1007/978-3-662-03539-9}}.

\bibitem{basar}
Tamer Ba\c{s}ar and Geert~Jan Olsder.
\newblock {\em Dynamic Noncooperative Game Theory}.
\newblock Society for Industrial and Applied Mathematics, 2nd edition, 1998.
\newblock \href {http://dx.doi.org/10.1137/1.9781611971132}
  {\path{doi:10.1137/1.9781611971132}}.

\bibitem{pauly-handbook}
Vasco Brattka, Guido Gherardi, and Arno Pauly.
\newblock Weihrauch complexity in computable analysis.
\newblock arXiv 1707.03202, 2017.

\bibitem{croft}
H.~T. Croft.
\newblock \enquote{{L}ion and {M}an} : A postscript.
\newblock {\em Journal of the London Mathematical Society}, s1-39(1):385--390,
  1964.
\newblock \href {http://dx.doi.org/10.1112/jlms/s1-39.1.385}
  {\path{doi:10.1112/jlms/s1-39.1.385}}.

\bibitem{denef}
J.~Denef and L.~Lipshitz.
\newblock Decision problems for differential equations.
\newblock {\em Journal of Symbolic Logic}, 54(3):941--950, 09 1989.
\newblock \href {http://dx.doi.org/10.2307/2274755}
  {\path{doi:10.2307/2274755}}.

\bibitem{dockner}
Engelbert Dockner, Steffen Jorgensen, Ngo Long, and Gerhard Sorger.
\newblock {\em Differential Games in Economics and Management Science}.
\newblock Cambridge University Press, 2000.
\newblock \href {http://dx.doi.org/10.1017/CBO9780511805127}
  {\path{doi:10.1017/CBO9780511805127}}.

\bibitem{flynn}
James Flynn.
\newblock Lion and man: The boundary constraint.
\newblock {\em SIAM Journal on Control}, 11(3):397--411, 1973.
\newblock \href {http://dx.doi.org/10.1137/0311032}
  {\path{doi:10.1137/0311032}}.

\bibitem{fomin}
Fedor~V. Fomin, Petr~A. Golovach, Jan Kratochv{\'i}l, Nicolas Nisse, and Karol
  Suchan.
\newblock Pursuing a fast robber on a graph.
\newblock {\em Theoretical Computer Science}, 411(7):1167 -- 1181, 2010.
\newblock \href {http://dx.doi.org/10.1016/j.tcs.2009.12.010}
  {\path{doi:10.1016/j.tcs.2009.12.010}}.

\bibitem{afriedman}
Avner Friedman.
\newblock {\em Differential Games}.
\newblock Wiley, New York, 1971.

\bibitem{hainry}
Emmanuel Hainry.
\newblock Reachability in linear dynamical systems.
\newblock In Arnold Beckmann, Costas Dimitracopoulos, and Benedikt L{\"o}we,
  editors, {\em CiE 2008: Logic and Theory of Algorithms}, pages 241--250.
  Springer, 2008.
\newblock \href {http://dx.doi.org/10.1007/978-3-540-69407-6\_28}
  {\path{doi:10.1007/978-3-540-69407-6\_28}}.

\bibitem{isaacs}
R.~Isaacs.
\newblock {\em Differential games}.
\newblock John Wiley and Sons, New York, 1965.

\bibitem{ivanov}
R.~P. Ivanov.
\newblock Simple pursuit-evasion on a compact set.
\newblock {\em Doklady Akademii Nauk SSSR}, 254(6):1318--1321, 1980.

\bibitem{suri}
Kyle Klein and Subhash Suri.
\newblock Catch me if you can: Pursuit and capture in polygonal environments
  with obstacles.
\newblock In {\em Proceedings of the Twenty-Sixth AAAI Conference on Artificial
  Intelligence}, AAAI'12, pages 2010--2016. AAAI Press, 2012.
\newblock URL: \url{http://dl.acm.org/citation.cfm?id=2900929.2901012}.

\bibitem{kopparty}
Swastik Kopparty and Chinya~V. Ravishankar.
\newblock A framework for pursuit evasion games in {$\mathbb{R}^n$}.
\newblock {\em Information Processing Letters}, 96(3):114 -- 122, 2005.
\newblock \href {http://dx.doi.org/10.1016/j.ipl.2005.04.012}
  {\path{doi:10.1016/j.ipl.2005.04.012}}.

\bibitem{subbotin}
N.~N. Krasovskiy and A.~I. Subbotin.
\newblock {\em Game-theoretical control problems}.
\newblock Springer, Berlin, 1988.

\bibitem{littlewood}
J.~E. Littlewood.
\newblock {\em A Mathematician's Miscellany}.
\newblock Methuen Co., London, 1957.

\bibitem{parsons}
T.~D. Parsons.
\newblock Pursuit-evasion in a graph.
\newblock In Yousef Alavi and Don~R. Lick, editors, {\em Theory and
  Applications of Graphs}, pages 426--441, Berlin, Heidelberg, 1978. Springer.
\newblock \href {http://dx.doi.org/10.1007/BFb0070400}
  {\path{doi:10.1007/BFb0070400}}.

\bibitem{parsons2}
T.~D. Parsons.
\newblock The search number of a connected graph.
\newblock In {\em Proceedings of the 10th Southeastern Conference on
  Combinatorics, Graph Theory and Computing}, pages 549--554, 1978.

\bibitem{pauly-synthetic}
Arno Pauly.
\newblock On the topological aspects of the theory of represented spaces.
\newblock {\em Computability}, 5(2):159--180, 2016.
\newblock URL: \url{http://arxiv.org/abs/1204.3763}, \href
  {http://dx.doi.org/10.3233/COM-150049} {\path{doi:10.3233/COM-150049}}.

\bibitem{petrosyan}
L.A. Petrosyan.
\newblock Survival differential game with many participants.
\newblock {\em Dokladi AN SSSR}, 161(2):285--287, 1965.

\bibitem{petrosyan2}
L.A. Petrosyan.
\newblock {\em Differential games of pursuit}.
\newblock Leningrad University Press, 1977.
\newblock in Russian.

\bibitem{pontryagin2}
L~S Pontryagin.
\newblock On the theory of differential games.
\newblock {\em Russian Mathematical Surveys}, 21(4):193--246, aug 1966.
\newblock \href {http://dx.doi.org/10.1070/rm1966v021n04abeh004171}
  {\path{doi:10.1070/rm1966v021n04abeh004171}}.

\bibitem{rado2}
Richard Rado.
\newblock How the lion tamer was saved.
\newblock {\em Math. Spectrum}, 6(1):14--18, 1973.

\bibitem{tsokos}
K.~M. Ramachandran and C.~P. Tsokos.
\newblock {\em Stochastic differential games}.
\newblock Atlantis Studies in Probability and Statistics. Springer, 2012.

\bibitem{rote}
G\"{u}nter Rote.
\newblock Pursuit-evasion with imprecise target location.
\newblock In {\em Proceedings of the Fourteenth Annual ACM-SIAM Symposium on
  Discrete Algorithms}, SODA '03, pages 747--753, Philadelphia, PA, USA, 2003.
  SIAM.
\newblock URL: \url{http://dl.acm.org/citation.cfm?id=644108.644231}.

\bibitem{satimov}
N~Satimov and A~Kuchkarov.
\newblock Deviation from encounter with several pursuers on a surface.
\newblock {\em Uzbek. Mat. Zh}, 1:51--55, 2001.

\bibitem{sgall}
Ji\v{r}\'{i} Sgall.
\newblock Solution of {D}avid {G}ale's lion and man problem.
\newblock {\em Theoretical Computer Science}, 259(1):663 -- 670, 2001.
\newblock \href {http://dx.doi.org/10.1016/S0304-3975(00)00411-4}
  {\path{doi:10.1016/S0304-3975(00)00411-4}}.

\bibitem{soare2}
Robert~I. Soare.
\newblock {\em Recursively Enumerable Sets and Degrees. A Study of Computable
  Functions and Computably Generated Sets}.
\newblock Springer, Berlin--New York, 1987.

\bibitem{weihrauchd}
Klaus Weihrauch.
\newblock {\em Computable Analysis}.
\newblock Springer-Verlag, 2000.

\bibitem{pachter}
Y.~Yavin, M.~Pachter, and Ervin Rodin, editors.
\newblock {\em Pursuit-Evasion differential games}. Pergamon Press, 1987.

\end{thebibliography}

\end{document}